\documentclass{amsart}

\usepackage{amssymb,amsmath,latexsym,amsthm}%or any usual package
\usepackage[english]{babel}
\usepackage[lowtilde]{url}

\newtheorem{thm}{Theorem}[section]

\newtheorem{lem}[thm]{Lemma}

\theoremstyle{definition}

\numberwithin{equation}{section}

\begin{document}

\title[cyclic torsion of elliptic curves]{On the cyclic torsion of elliptic curves over cubic number fields (III)}

%first author

\author{Jian Wang}
\address{Jian Wang\\
College of Mathematics\\ Jilin Normal University\\
Siping, Jilin 136000\\
China}
\email{blandye@gmail.com}
%\urladdr{http://www....}

% second author
%\author{Firstname2 Lastname2}
%\address{Firstname2 Lastname2\\
%address\\
%country}
%\email{email address}
%\urladdr{http://www....}

\date{\today}



%\thanks{We thank Sheldon Kamienny for valuable ideas and insightful comments. We thank the anonymous referee for helpful comments and suggestions.}

\maketitle

%\begin{resume}
%C¡¯est la troisi\`{e}me partie d¡¯une s\'{e}rie de documents traitant des
%le sous-groupe de torsion cyclique des courbes elliptiques sur le nombre cubique
%des champs. Pour $N=39,42$, nous montrons que $\mathbb{Z}/N\mathbb{Z}$ n'est pas un sous-groupe de
%$E(K)_{tor}$ pour toute courbe elliptique $E$ sur un corps de nombres cubiques $K$.
%\end{resume}

\begin{abstract}
This is the third part of a series of papers discussing the cyclic torsion subgroup of elliptic curves over cubic number fields. For $N=39$, we show that $\mathbb{Z}/N\mathbb{Z}$ is not a subgroup of $E(K)_{tor}$ for any elliptic curve $E$ over a cubic number field $K$.
\end{abstract}

\bigskip
%\section{Introduction}

In \cite{Wang18}, the author conjectured that $\mathbb{Z}/N\mathbb{Z}$ is not a cyclic torsion subgroup of the Mordell-Weil group of any elliptic curve over a cubic number field for 24 values of $N$:
$$22, 24, 25, 26, 27, 28, 30, 32, 33, 35, 36, 39, 40, $$$$42, 45, 49, 55, 63, 65, 77, 91, 121, 143, 169.$$

The cases $N=55, 65, 77, 91, 143, 169$ were proved in the first part \cite{Wang18} of this series of papers using refinements of a criterion originally due to Kamienny. The cases $N=22,25,40,49$ were proved in the second part \cite{Wang20} of this series of papers with the help of Kato's result about the Birch and Swinnerton-Dyer conjecture on modular abelian varieties. In this paper, we prove the $N=39$ case, which is Theorem \ref{T1}.

Let $N$ be a positive integer. Throughout this paper, $X_1(N)$ (resp. $X_0(N)$) denotes the modular curve over $\mathbb{Q}$ associated to the congruence subgroup $\Gamma_1(N)$ (resp. $\Gamma_0(N)$). $Y_1(N)$ (resp. $Y_0(N)$) denotes the corresponding affine curve without cusps. $J_1(N)$ (resp. $J_0(N)$) denotes the Jacobian of $X_1(N)$ (resp. $X_0(N)$).

Let $X_0(N)^{(d)}$ be the $d$-th symmetric power of $X_0(N)$. Suppose $K$ is a number field of degree $d$ over $\mathbb{Q}$ and $x\in X(K)$. Let $x_1,\cdots,x_d$ be the images of $x$ under the distinct embeddings $\tau_i:K\hookrightarrow\mathbb{C},1\leq i\leq d$. Define $$\Phi:X_0(N)^{(d)}\longrightarrow J_0(N)$$
by $\Phi(P_1+\cdots+P_d)=[P_1+\cdots+P_d-d\infty]$ where $[~~]$ denotes the divisor class.

Any point of $Y_1(N)$ is represented by $(E,\pm P)$, where $E$ is an elliptic curve and $P\in E$ is a point of order $N$. Any point of $Y_0(N)$ is represented by $(E, C)$, where $E$ is an elliptic curve and $C\subset E$ is a cyclic subgroup of order $N$. The covering map $\pi: X_1(N)\longrightarrow X_0(N)$ sends $(E,\pm P)$ to $(E,\langle P\rangle)$, where $\langle P\rangle$ is the cyclic subgroup generated by $P$.

Let $K$ be a number field with ring of integers $\mathcal{O}_K$, $\wp\subset\mathcal{O}_K$ a prime ideal lying above $p$, $k=\mathbb{F}_q=\mathcal{O}_K/\wp$ its residue field. Let $A$ be an abelian variety over $K$ and $P\in A(K)$ a point of order $N$. Let $\widetilde{A}$ be the fibre over $k$ of the N\'{e}ron model of $A$, and let $\widetilde{P}\in\widetilde{A}(k)$ be the reduction of $P$. The following lemma (see, for instance, \cite[\S 7.3 Proposition 3]{BoschLutkebohmertRaynaud}) shows that $\widetilde{P}$ has order $N$ when $p\nmid N$.

\begin{lem}\label{torsionreduction} Let $m$ be a positive integer relatively prime to $char(k)$. Then the reduction map
$$A(K)[m]\longrightarrow\widetilde{A}(k)$$
is injective.
\end{lem}

For a square-free $N$, the order $h_0(N)$ of $J_0(N)^c$ can be calculated by the following formula due to Takagi \cite{Takagi97}:
\begin{equation}\label{Takagi}h_0(N)=\frac{2^f12a_2a_3}{d}\prod_{\chi\neq1}\left(\frac{1}{24}\prod_{p|N}(p+\chi([p]))\right)\end{equation}
where $f$ is the number of the prime factors of $N=p_1\cdots p_f$, $d=(12,p_1-1,\cdots,p_f-1)$,
$a_2=2$ if $2|N$ and $N$ has a prime factor $p$ with $p\equiv3\mod4$, $a_2=1$ otherwise,
$a_3=3$ if $3|N$ and $N$ has a prime factor $p$ with $p\equiv2\mod3$, $a_3=1$ otherwise, $\chi$ runs through all nontrivial characters of the group $T$ (consisting of all the positive divisors of $N$ and isomorphic to $(\mathbb{Z}/2\mathbb{Z})^f$ ) and $p$ runs through all prime factors of $N$. 

Of the remaining 14 cases of $N$, 6 are square-free. The order $h_0(N)$ of $J_0(N)^c$ is calculated and listed in Table \ref{order}.

\begin{table}[!ht]
\tabcolsep 0pt
\vspace*{0pt}
\begin{center}
\def\temptablewidth{1\textwidth}
\setlength{\abovecaptionskip}{0pt}
\setlength{\belowcaptionskip}{-5pt}
\caption{The order $h_0(N)$ of $J_0(N)^c$}
\label{order}
{\rule{\temptablewidth}{1pt}}
\begin{tabular*}{\temptablewidth}{@{\extracolsep{\fill}}ccccccccccccc}
~~~~$N$~~~~&$26$&$30$&$33$&$35$&$39$&$42$\\\hline
~~~~$h_0(N)$~~~~&$3\cdot7$&$2^6\cdot3$&$2^2\cdot5^2$&$2^4\cdot3$&$2^3\cdot7$&$2^8\cdot3^2$\\\hline

\end{tabular*}
{\rule{\temptablewidth}{1pt}}
\end{center}
\end{table}

%% Note that in the example below, the braces { } around \cite are necessary (due to nested optional parameters)

The following specialization lemma follows from the classification of Oort-Tate \cite{OortTate} on finite flat group schemes of rank $p$ (or more generally the classification of finite flat group schemes of type $(p,\cdots,p)$ by Raynaud \cite{Raynaud}). If the group scheme is contained in an abelian variety, this lemma follows from elementary properties of formal Lie groups (see, for example, the Appendix of Katz\cite{Katz}).

\begin{lem}[Specialization Lemma]\label{Katz} Let $K$ be a number field. Let $\wp\subset\mathcal{O}_K$ be a prime above $p$. Let $A/K$ be an abelian variety. Suppose the ramification index $e_\wp(K/\mathbb{Q})<p-1$. Then the reduction map
$$\Psi: A(K)_{tor}\longrightarrow A(\overline{\mathbb{F}}_p)$$
is injective.
\end{lem}

Using the same method as that in \cite{Wang20} to show the finiteness of $J_1(N)(\mathbb{Q})$, we can verify that $J_0(N)(\mathbb{Q})$ is finite for certain $N$. For the 24 cases we are interested in, Table \ref{jacobian} is the result of calculations in Magma \cite{Magma}. The second column $t$ is the number of non-isogenous modular abelian varieties in the decomposition $J_0(N)=\bigoplus_{i=1}^tA_i^{m_i}$. The third column list the dimension $d_i$ and multiplicity $m_i$ of each $A_i$ (we omit $m_i$ if $m_i=1$). The fourth column verifies non-vanishing of $L$-series at $1$ (a mark $T$ means $L(A_i,1)\neq0$ is verified, otherwise we place a mark $\textbf{\emph{F}}$).

\begin{table}[!ht]
\tabcolsep 0pt
\vspace*{0pt}
\begin{center}
\def\temptablewidth{1\textwidth}
\setlength{\abovecaptionskip}{0pt}
\setlength{\belowcaptionskip}{-5pt}
\caption{Decompostion of $J_0(N)$}
\label{jacobian}
{\rule{\temptablewidth}{1pt}}
\begin{tabular*}{\temptablewidth}{@{\extracolsep{\fill}}ccccccccccccc}
~~~~$N$~~~~&$t$&$d_i(m_i)$&$L(A_i,1)\neq0$\\\hline
$169$&$3$&$2,3,3$&$T,\textbf{\emph{F}},T$\\
$121$&$6$&$1,1,1,1,1,1$&$\textbf{\emph{F}},T,T,T,T,T$\\
$49$&$1$&$1$&$T$\\
$25$&$1$&$0$&$-$\\
$27$&$1$&$1$&$T$\\
$32$&$1$&$1$&$T$\\
$143$&$5$&$1,4,6,1,1$&$\textbf{\emph{F}},T,T,T,T$\\
$91$&$4$&$1,1,2,3$&$\textbf{\emph{F}},\textbf{\emph{F}},T,T$\\
$65$&$3$&$1,2,2$&$\textbf{\emph{F}},T,T$\\
$39$&$2$&$1,2$&$T,T$\\
$26$&$2$&$1,1$&$T,T$\\
$77$&$6$&$1,1,1,2,1,1$&$\textbf{\emph{F}},T,T,T,T,T$\\
$55$&$4$&$1,2,1,1$&$T,T,T$\\
$33$&$3$&$1,1,1$&$T,T,T$\\
$22$&$2$&$1,1$&$T,T$\\
$35$&$2$&$1,2$&$T,T$\\
$63$&$4$&$1,2,1,1$&$T,T,T,T$\\
$42$&$5$&$1,1,1,1,1$&$T,T,T,T,T$\\
$28$&$2$&$1,1$&$T,T$\\
$45$&$3$&$1,1,1$&$T,T,T$\\
$30$&$3$&$1,1,1$&$T,T,T$\\
$40$&$3$&$1,1,1$&$T,T,T$\\
$36$&$1$&$1$&$T$\\
$24$&$1$&$1$&$T$\\

\end{tabular*}
{\rule{\temptablewidth}{1pt}}
\end{center}
\end{table}

\begin{thm}\label{T1} If $N=39$, then $\mathbb{Z}/N\mathbb{Z}$ is not a subgroup of $E(K)_{tor}$ for any elliptic curve $E$ over a cubic number field $K$.
\end{thm}

\begin{proof}

Let $p=3$ and $N'=N/p$. Let $K$ be a cubic field and $\wp$ a prime of $K$ over $p$. As shown in the proof of Lemma 3.6 of \cite{Wang18}, we can always choose $\wp$ such that the residue field $k=\mathcal{O}_K/\wp$ has degree $1$ or $3$ over $\mathbb{F}_p$. Suppose $x=\pi(E,\pm P)\in Y_0(N)(K)$, then $E$ cannot have additive reduction since $N>4$.

If $k$ has degree $3$, then $e_\wp(K/\mathbb{Q})=1<p-1$, by Lemma \ref{Katz}, $E$ has multiplicative reduction since $N>(1+\sqrt{3^3})^2$. If $k$ has degree $1$, by Lemma \ref{torsionreduction}, $E$ also has multiplicative reduction since $N'>(1+\sqrt{3})^2$.

Since $E$ has multiplicative reduction at $\wp$, then $x$ specializes to a cusp of $\widetilde{X}_0(N)$. Recall the notation of $\tau_i$ and $x_i$, $1\leq i\leq 3$. Then $\tau_i(K)$ is also a cubic field with prime ideal $\tau_i(\wp)$ over $p$ and residue field $k_i=k$. And $\tau_i(E)$ also has multiplicative reduction at $\tau_i(\wp)$. This means all the images $x_1,x_2, x_3$ of $x$ specialize to cusps of $\widetilde{X}_0(N)$. Let $c_1,c_2, c_3$ be the cusps such that
$$x_i\otimes\overline{\mathbb{F}}_p=c_i\otimes\overline{\mathbb{F}}_p, ~~~~~1\leq i\leq 3$$

We know all the cusps of $X_0(N)$ are defined over $\mathbb{Q}(\zeta_N)$ \cite{Ogg73}. As is seen in Table \ref{order}, the order of $J_0(N)^c$ is $2^3\cdot7$. Therefore by Lemma \ref{torsionreduction}, the specialization map
$$\Psi: J_0(N)(\mathbb{Q}(\zeta_N))_{tor}\longrightarrow J_0(N)(\overline{\mathbb{F}}_p)$$
is injective.

We know $x_1+x_2+x_3$ is $\mathbb{Q}$-rational. The data in Table \ref{jacobian} show the finiteness of $J_0(N)(\mathbb{Q})$. So $[x_1+x_2+x_3-3\infty]$  is in $J_0(N)(\mathbb{Q}(\zeta_N))_{tor}$. By a theorem of Manin \cite{Manin} and Drinfeld \cite{Drinfeld}, the difference of two cusps of $X_0(N)$ has finite order in $J_0(N)$. So $[c_1+c_2+c_3-3\infty]$ is also in $J_0(N)(\mathbb{Q}(\zeta_N))_{tor}$. Therefore $$\Psi([x_1+x_2+x_3-3\infty])=\Psi([c_1+c_2+c_3-3\infty])$$ implies $$[x_1+x_2+x_3-3\infty]=[c_1+c_2+c_3-3\infty]$$

Since $X_0(N)$ is not trigonal \cite{HasegawaShimura}, then similar reasoning as in the proof of Proposition 1 in Frey \cite{Frey} shows that $x_1+x_2+x_3=c_1+c_2+c_3$. This is a contradiction because we assume $x$ is a noncuspidal point.

\end{proof}

\bigskip

%{\bf Please use the numerical style for the bibliography.}

\end{document}